\documentclass[12pt]{amsart}
\usepackage{amssymb}
\usepackage{amsmath}
\usepackage{amsfonts}
\usepackage{amsthm}
\usepackage{lipsum}

\newtheorem{thm}{Theorem}[section]
\newtheorem{thmx}{Theorem}

\newtheorem{cor}[thm]{Corollary}
\newtheorem{lem}[thm]{Lemma}
\newtheorem{prop}[thm]{Proposition}
\theoremstyle{definition}

\theoremstyle{remark}
\newtheorem{rem}[thm]{Remark}
\theoremstyle{example}

\numberwithin{equation}{section}

\begin{document}
\title[Determination of a jump]{Determination of a jump by Fourier and Fourier-Chebyshev series}
\author{Muharem Avdispahi\'{c}, Zenan \v{S}abanac}
\address{Department of Mathematics, University of Sarajevo, Zmaja od Bosne 33-35, 71000 Sarajevo, Bosnia and Herzegovina}
\email{mavdispa@pmf.unsa.ba, zsabanac@pmf.unsa.ba}
\maketitle

\begin{abstract}
By observing the equivalence of assertions on determining the jump of a
function by its differentiated or integrated Fourier series, we generalize a
previous result of Kvernadze, Hagstrom and Shapiro to the whole class of
functions of harmonic bounded variation and without finiteness assumption on
the number of discontinuities. Two results on determination of jump
discontinuities by means of the tails of integrated Fourier-Chebyshev series are derived.
\end{abstract}

\section{Introduction}

\label{intro} \ \ The problem of approximating the magnitudes of jumps of a
function by means of its truncated Fourier series arises naturally from the
attempt to overcome the Gibbs phenomenon which describes the characteristic
oscillatory behaviour of the Fourier partial sums of a piecewise smooth
function in the neighbourhood of a point of discontinuity. The identity
determining the jumps of a function of bounded variation ($BV$) using
partial sums of its differentiated Fourier series has been known for a long
time. The equation
\begin{equation}
\underset{n\rightarrow \infty }{\lim }\frac{S_{n}^{^{\prime }}(f,x)}{n}=%
\frac{1}{\pi }\left[ f\left( x+0\right) -f\left( x-0\right) \right] ,
\label{eq0}
\end{equation}%
where $S_{n}^{^{\prime }}(f,x)$ denotes the $n$th partial sum of the
differentiated Fourier series of a function $f$ at a point $x$, was proved
by L. Fej\'{e}r \cite{FE} for $f$ satisfying the so-called
Dirichlet-condition, by P. Csillag \cite{CS} for functions of bounded
variation and by B. I. Golubov \cite[Theorem 1, p. 20]{GO} for functions in $%
V_{p}$, $1\leq p<\infty $, of Wiener's bounded variation. M. Avdispahi\'{c}
\cite[Theorem 1, p. 268]{AV3} has shown that equation \eqref{eq0} holds for
any function $f\in HBV$ and that $HBV$ is the limiting case in the scale of $%
\Lambda BV$ spaces for validity of \eqref{eq0}. The corresponding formula
involving the partial sums of the conjugate Fourier series of $f\in HBV$ is
also derived there. A number of results from \cite{AV1, AV2, AV3, AV4}
related to the classes $V_{\phi }$, $\Lambda BV$ and $V[\nu ]$ were later
rediscovered and differently proved in \cite{KV1}. G. Kvernadze \cite{KV1}
extended \cite[Theorem 1' (1)]{AV3} to the setting of the generalized
Fourier-Jacobi series.

Special formulae which determine the jumps of a $2\pi -$periodic function in
$V_{p}$, $1\leq p<2$, with a finite number of jump discontinuities, by means
of the tails of its integrated Fourier series have been established by G.
Kvernadze, T. Hagstrom and H. Shapiro in \cite{KHS}.

This paper consists of two main parts. In the first part, we generalize a
result of \cite{KHS} to the whole class of functions of harmonic bounded
variation and without finiteness assumption on the number of
discontinuities. New results on determination of jump discontinuities by
means of integrated Fourier-Chebyshev series in case of $HBV$ and $%
V_{2}$ spaces are presented in the second part.

\section{Jump of a $HBV$\ function and integrated Fourier series}

\subsection{Classes of functions of bounded variation}

\ \ The classes of functions of bounded variation of higher orders were
firstly introduced by N. Wiener \cite{WI}. A function $f$ is said to be of%
\textit{\ bounded }$p-$\textit{variation} on $\left[ 0,2\pi \right] $, $%
p\geq 1$, and belongs to the class $V_{p}$ if%
\begin{equation*}
V_{p}(f)=\sup \left\{ \underset{i}{\sum }\left\vert f(I_{i})\right\vert
^{p}\right\} ^{1/p}<\infty \text{,}
\end{equation*}%
where the supremum is taken over all finite collections of nonoverlapping
subintervals $I_{i}$ of $\left[ 0,2\pi \right] $. The quantity $V_{p}(f)$ is
called the\textit{\ }$p-$\textit{variation of }$f$\textit{\ }on $\left[
0,2\pi \right] $.

This concept has been generalized by L. C. Young \cite{YO}. Let $\phi $ be a
continuous function defined on $\left[ 0,\infty \right) $ and strictly
increasing from $0$ to $\infty $. A function $f$ is said to be of\textit{\
bounded }$\phi -$\textit{variation }on $\left[ 0,2\pi \right] $ and belongs
to the class $V_{\phi }$ if%
\begin{equation*}
V_{\phi }(f)=\sup \left\{ \underset{i}{\sum }\phi \left( \left\vert
f(I_{i})\right\vert \right) \right\} <\infty \text{,}
\end{equation*}%
where the supremum is taken over all finite collections of nonoverlapping
subintervals $I_{i}$ of $\left[ 0,2\pi \right] $. The quantity $V_{\phi }(f)$
is called the\textit{\ }$\phi -$\textit{variation of }$f$\textit{\ }on $%
\left[ 0,2\pi \right] $.

By taking $\phi \left( u\right) =u$ we get Jordan's class $BV$, while $\phi
\left( u\right) =u^{p}$ gives Wiener's class $V_{p}$.

Another type of generalization of the class $BV$, influenced by a joint work
with C. Goffman on everywhere convergence of Fourier series, was introduced
by D. Waterman in \cite{WA}. Let $\Lambda =\left\{ \lambda _{n}\right\} $ be
a nondecreasing sequence of positive numbers tending to infinity, such that $%
\sum 1/\lambda _{n}$ diverges. A function $f$ is said to be of\textit{\
bounded }$\Lambda -$\textit{variation} on $\left[ 0,2\pi \right] $ and
belongs to the class $\Lambda BV$ if%
\begin{equation*}
V_{\Lambda }(f)=\sup \left\{ \underset{i}{\sum }\left\vert
f(I_{i})\right\vert /\lambda _{i}\right\} <\infty \text{,}
\end{equation*}%
where the supremum is taken over all finite collections of nonoverlapping
subintervals $I_{i}$ of $\left[ 0,2\pi \right] $. The quantity $V_{\Lambda
}(f)$ is called the\textit{\ }$\Lambda -$\textit{variation of }$f$\textit{\ }%
on $\left[ 0,2\pi \right] $. In the case when $\Lambda =\left\{ n\right\} $,
the sequence of positive integers, the function $f$ is said to be of \textit{%
harmonic bounded variation} and the corresponding class is denoted by $HBV$.

By $W$ we denote the class of \textit{regulated functions}, i.e. functions
possesing the one-sided limits at each point. $W$ is the union of all $%
\Lambda BV$ spaces \cite{PE}.

Z. Chanturiya \cite{CH} gave another interesting generalization of bounded
variation using the modulus of variation. The \textit{modulus of variation}
of a bounded function $f$ is the function $\nu _{f}$ whose domain is the set
of positive integers, given by%
\begin{equation*}
\nu _{f}\left( n\right) =\underset{\Pi _{n}}{\sup }\left\{ \underset{k=1}{%
\overset{n}{\sum }}\left\vert f(I_{k})\right\vert \right\} \text{,}
\end{equation*}%
where $\Pi _{n}=\left\{ I_{k}:k=1,\ldots ,n\right\} $ is an arbitrary finite
collection of $n$ nonoverlapping subintervals of $\left[ 0,2\pi \right] $.
The modulus of variation of any bounded function is nondecreasing and
concave. Given a function $\nu $ whose domain is the set of positive
integers with such properties, then by $V\left[ \nu \right] $ one denotes
the class of functions $f$ for which $\nu _{f}\left( n\right) =O\left( \nu
\left( n\right) \right) $ as $n\rightarrow \infty $. We note that $V_{\phi
}\subseteq V\left[ n\phi ^{-1}\left( 1/n\right) \right] $ and $W=\left\{
f:\nu _{f}\left( n\right) =o\left( n\right) \right\} $ \cite{CH}.

There exist the following inclusion relations between Wiener's, Waterman's
and Chanturiya's classes.

\begin{thmx}[cf. Theorem 4.4. in \protect\cite{AV1}]
\label{AVth}
\begin{equation*}
\left\{ n^{\alpha }\right\} BV\subset V_{\frac{1}{1-\alpha }}\subset V\left[
n^{\alpha }\right] \subset \left\{ n^{\beta }\right\} BV\text{,}
\end{equation*}%
for $0<\alpha <\beta <1$.
\end{thmx}

\subsection{Ces\`aro summability and differentiated Fourier series}

\ \ As well known, a sequence $\left\{s_{n}\right\} $ is \textit{Ces\`aro }or%
\textit{\ }$\left( C,1\right) $\textit{\ summable }to $s$ if the sequence $%
\left\{ \sigma _{n}\right\} $ of its arithmetical means converges to $s$,
i.e.
\begin{equation*}
\sigma _{n}=\frac{s_{0}+s_{1}+\cdots +s_{n}}{n+1}\rightarrow s, \text{ \ \ }%
n\rightarrow \infty \text{.}
\end{equation*}

Analogously, the sequence $\left\{s_{n}\right\} $ is $(C,\alpha )$, $\alpha
>-1$, summable to $s$, if the sequence
\begin{equation*}
\sigma _{n}^{(\alpha )}=\frac{1}{\binom{n+\alpha }{n}}\sum_{i=0}^{n}\binom{%
n-i+\alpha -1}{n-i}s_{i}
\end{equation*}%
converges to $s$.

It is obvious that Fej\'{e}r's identity \eqref{eq0} for determination of the
jump of a function $f\in BV$ is equivalent to Ces\`aro summability of the
sequence $\left\{ kb_{k}\cos kx-ka_{k}\sin kx\right\} $, where $%
a_{k}=a_{k}\left( f\right) $ and $b_{k}=b_{k}\left( f\right) $ are the $k$th
cosine and sine coefficient of the Fourier series of the function $f$,
respectively. There exist numerous generalizations of Fej\'{e}r's theorem to
more general summability methods. We recall the relationship between the
order of Ces\`aro summability of the sequence $\left\{ kb_{k}\cos
kx-ka_{k}\sin kx\right\} $ and the "order of variation" of a function $f$.

\begin{thmx}[\protect\cite{AV3}, \protect\cite{AV4}, \protect\cite{AV5}]
The sequence \bigskip $\left\{ kb_{k}\cos kx-ka_{k}\sin kx\right\} $ of the
terms of the differented Fourier series of a function $f$ of generalized
bounded variation is $(C,\alpha )$ summable to $\frac{1}{\pi }\left[ f\left(
x+0\right) -f\left( x-0\right) \right]$ at every point $x$ for

\begin{enumerate}
\item $\alpha >0$, if $f\in BV$,

\item $\alpha >1-\frac{1}{p}$, if $f\in V_{p}$, $1<p<\infty $,

\item $\alpha >\beta $, if $f\in V\left[ n^{\beta }\right] $, $0<\beta <1$,

\item $\alpha =1$, if $f\in HBV$,

\item $\alpha >1$, if $f\in W$.
\end{enumerate}
\end{thmx}

\subsection{Jump of a function and integrated Fourier series}

\ \ A method of determining jumps of a $2\pi -$periodic function by means of
the tails of its integrated Fourier series was introduced in \cite{KHS},
where special formulae to determine the jumps of a $2\pi -$periodic function
of $V_{p}$, $1\leq p<2$, class, with a finite number of jump
discontinuities, were derived.

For any function $f$, integrable on $\left[ -\pi ,\pi \right] $, we define $%
f^{(-r)}$, $r\in
%TCIMACRO{\U{2115} }%
%BeginExpansion
\mathbb{N}
%EndExpansion
_{0}$, as%
\begin{equation*}
f^{\left( -r-1\right) }\equiv \int f^{\left( -r\right) }\text{,}
\end{equation*}%
where $f^{\left( 0\right) }\equiv f$, and the constants of integration are
successively determined by the condition%
\begin{equation*}
\int_{-\pi }^{\pi }f^{\left( -r\right) }(t)dt=0\text{, }r\in
%TCIMACRO{\U{2115} }%
%BeginExpansion
\mathbb{N}
%EndExpansion
_{0}\text{.}
\end{equation*}

We generalize a result of \cite[Theorem 4., p. 32]{KHS} to the whole class
of $HBV$ functions and without finiteness assumption on the number of
discontinuities. The result is presented in the following theorem.

\begin{thm}
\label{aaa}\mbox{}\\*[0pt]
\begin{enumerate}
\item[a)] Let $g\in HBV$ and $r=0,$ $1,$ $2,$ \ldots . Then, for any point $%
x_{0}$ we have%
\begin{equation}
\underset{n\rightarrow \infty }{\lim }n^{2r+1}R_{n}^{\left( -2r-1\right)
}\left( g,x_{0}\right) =\frac{\left( -1\right) ^{r+1}}{\left( 2r+1\right)
\pi }\left[ g\left( x_{0}+0\right) -g\left( x_{0}-0\right) \right] ,
\label{1}
\end{equation}%
where $R_{n}\left( g,x\right) $ denotes the $n$th order tail of the Fourier
series of $g$, i.e.%
\begin{equation*}
R_{n}\left( g,x\right) =\underset{k=n}{\overset{\infty }{\sum }}\left(
a_{k}\left( g\right) \cos kx+b_{k}\left( g\right) \sin kx\right) \text{.}
\end{equation*}

\item[b)] If $\Lambda $ is such that $\Lambda BV\supsetneqq HBV$, the
assertion a) does not hold for $\Lambda BV\setminus HBV$.
\end{enumerate}
\end{thm}

\begin{proof}
\begin{enumerate}
\item[a)]
\end{enumerate}
Let $g\in HBV$ and $S_{n}^{{^{\prime }}}\left( g,x_{0}\right) =\underset{k=1}%
{\overset{n}{\sum }}\left( -ka_{k}\left( g\right) \sin kx_{0}+kb_{k}\left(
g\right) \cos kx_{0}\right) $. For brevity, we denote by $c\equiv c\left(
g,x_{0}\right) =\frac{1}{\pi }\left[ g\left( x_{0}+0\right) -g\left(
x_{0}-0\right) \right] $ the jump of the function $g$ at $x_{0}$ and put $%
A_{k}\equiv A_{k}\left( g,x_{0}\right) =a_{k}\left( g\right) \sin
kx_{0}-b_{k}\left( g\right) \cos kx_{0}$. By \cite[Theorem 1, p. 268]{AV3},
one has
\begin{equation*}
\underset{n\rightarrow \infty }{\lim }\frac{S_{n}^{{^{\prime }}}\left(
g,x_{0}\right) }{n}=\frac{1}{\pi }\left[ g\left( x_{0}+0\right) -g\left(
x_{0}-0\right) \right] \text{,}
\end{equation*}%
or equivalently%
\begin{equation}
s_{n}\equiv s_{n}\left( g,x_{0}\right) \equiv c+\frac{1}{n}\underset{k=1}{%
\overset{n}{\sum }}kA_{k}=o\left( 1\right) ,\text{ }n\rightarrow \infty
\text{.}  \label{2}
\end{equation}%
Multiplying (\ref{2}) by $n$ and rearranging the terms, we get%
\begin{equation}
ns_{n}=\underset{k=1}{\overset{n}{\sum }}\left( kA_{k}+c\right) =o\left(
n\right), \ n\rightarrow \infty .  \label{3}
\end{equation}%
Obviously,
\begin{equation}
ns_{n}-\left( n-1\right) s_{n-1}=nA_{n}+c \label{4}
\end{equation}%
and $R_{n}^{\left( -2r-1\right) }\left( g,x_{0}\right) =\underset{k=n}{%
\overset{\infty }{\sum }}\frac{\left( -1\right) ^{r}\left( a_{k}\left(
g\right) \sin kx_{0}-b_{k}\left( g\right) \cos kx_{0}\right) }{k^{2r+1}}%
=\left( -1\right) ^{r}\underset{k=n}{\overset{\infty }{\sum }}\frac{A_{k}}{%
k^{2r+1}}$. Now, it is enough to prove that
\begin{equation}
n^{2r+1}\underset{k=n}{\overset{\infty }{\sum }}\frac{A_{k}}{k^{2r+1}}%
\rightarrow -\frac{c}{2r+1}, \text{ \ \ }n\rightarrow \infty \text{.}  \label{5}
\end{equation}%
Using (\ref{2}), (\ref{3}), (\ref{4}), Abel's partial summation formula and
the fact that $\underset{n\rightarrow \infty }{\lim }n^{2r+1}\underset{k=n}{%
\overset{\infty }{\sum }}\frac{1}{k^{2r+2}}=\frac{1}{2r+1}$, we get
\begin{eqnarray*}
n^{2r+1}\underset{k=n}{\overset{\infty }{\sum }}\frac{A_{k}}{k^{2r+1}}
&=&n^{2r+1}\underset{k=n}{\overset{\infty }{\sum }}\frac{kA_{k}}{k^{2r+2}}%
=n^{2r+1}\underset{k=n}{\overset{\infty }{\sum }}\frac{ks_{k}-\left(
k-1\right) s_{k-1}-c}{k^{2r+2}} \\
&=&n^{2r+1}\underset{k=n}{\overset{\infty }{\sum }}\frac{ks_{k}-\left(
k-1\right) s_{k-1}}{k^{2r+2}}-cn^{2r+1}\underset{k=n}{\overset{\infty }{%
\sum }}\frac{1}{k^{2r+2}} \\
&=&n^{2r+1}\left\{ -\frac{\left( n-1\right) s_{n-1}}{n^{2r+2}}+\underset{k=n}%
{\overset{\infty }{\sum }}\left[ \frac{1}{k^{2r+2}}-\frac{1}{\left(
k+1\right) ^{2r+2}}\right] ks_{k}\right\} -cn^{2r+1}\underset{k=n}{\overset{%
\infty }{\sum }}\frac{1}{k^{2r+2}}.
\end{eqnarray*}
Notice that $s_k=o(1)$ and $\left[ \frac{1}{k^{2r+2}}-\frac{1}{\left(
k+1\right) ^{2r+2}}\right] k = \frac{(k+1)^{2r+2}-k^{2r+2}}{k^{2r+1}(k+1)^{2r+2}}=\frac{(2r+2)\xi_k^{2r+1}}{k^{2r+1}(k+1)^{2r+2}}$, where $\xi_k\in (k,k+1)$.
Thus,
\begin{eqnarray*}
n^{2r+1}\underset{k=n}{\overset{\infty }{\sum }}\frac{A_{k}}{k^{2r+1}}
&=&-\frac{\left( n-1\right) }{n}s_{n-1} +o\left(n^{2r+1}%
\underset{k=n}{\overset{\infty }{\sum }}\frac{1}{k^{2r+2}}\right)%
-cn^{2r+1}\underset{k=n}{\overset{\infty }{\sum }}\frac{1}{k^{2r+2}} \\
&\rightarrow &-\frac{c}{2r+1},\text{ \ \ }n\rightarrow \infty \text{,}
\end{eqnarray*}
and the proof is complete.
\begin{enumerate}
\item[b)]
\end{enumerate}
If $\Lambda $ is such that $\Lambda BV\supsetneqq HBV$, by \cite[Remark 4.,
p. 269]{AV3} there exists a continuous function $g\in \Lambda BV$ with the
property
\begin{equation}
\underset{k=1}{\overset{n}{\sum }}kA_{k}\neq O\left( n\right) .  \label{5a}
\end{equation}%
Suppose (\ref{5}) holds true for $g$ and some nonnegative integer $r$. Then,
denoting $\overset{\infty }{\underset{k=n}{\sum }}\frac{A_{k}}{k^{2r+1}}$ \
by $\sigma _{n}$, we get%
\begin{eqnarray*}
\underset{k=1}{\overset{n}{\sum }}kA_{k} &=&\underset{k=1}{\overset{n}{\sum }%
}k^{2r+2}\left( \sigma _{k}-\sigma _{k+1}\right) =\sigma _{1}+\underset{k=2}{%
\overset{n}{\sum }}\left( k^{2r+2}-(k-1)^{2r+2}\right) \sigma
_{k}-n^{2r+2}\sigma _{n+1}= \\
&=&\sigma _{1}+\underset{k=2}{\overset{n}{\sum }}\left(
(2r+2)k^{2r+1}+O\left( k^{2r}\right) \right) \sigma _{k}-n^{2r+2}\sigma
_{n+1}\text{.}
\end{eqnarray*}%
Hence,
\begin{equation*}
\frac{1}{n}\underset{k=1}{\overset{n}{\sum }}kA_{k}=\frac{\sigma _{1}}{n}+%
\frac{1}{n}\underset{k=2}{\overset{n}{\sum }}(2r+2)k^{2r+1}\sigma _{k}+\frac{%
1}{n}\underset{k=2}{\overset{n}{\sum }}O\left( k^{2r}\right) \sigma
_{k}-n^{2r+1}\sigma _{n+1}\text{.}
\end{equation*}%
Letting $n\rightarrow \infty $ and having in mind that%
\begin{gather*}
\frac{\sigma _{1}}{n}\rightarrow 0\text{, }\frac{1}{n}\underset{k=2}{\overset%
{n}{\sum }}(2r+2)k^{2r+1}\sigma _{k}\sim (2r+2)n^{2r+1}\sigma
_{n}\rightarrow -\frac{2r+2}{2r+1}c\text{, } \\
\frac{1}{n}\underset{k=2}{\overset{n}{\sum }}O\left( k^{2r}\right) \sigma
_{k}\sim \frac{1}{n}O\left( n^{2r+1}\right) \sigma _{n}\rightarrow 0\text{
and }n^{2r+1}\sigma _{n+1}\rightarrow -\frac{1}{2r+1}c\text{,}
\end{gather*}
we get%
\begin{equation*}
\frac{1}{n}\underset{k=1}{\overset{n}{\sum }}kA_{k}\rightarrow -c \text{.}
\end{equation*}%
This obviously contradicts (\ref{5a}).
\end{proof}

Making use of \cite[Theorem 1' (2)]{AV3} and following the same line of argumentation as in the proof of Theorem \ref{aaa}, one obtains

\begin{thm}
\label{aaa1}\mbox{}\\*[0pt]
\begin{enumerate}
\item[a)] Let $g\in HBV$ and $r=1,$ $2,$ \ldots . Then, for any point $%
x_{0}$ we have%
\begin{equation}
\underset{n\rightarrow \infty }{\lim }n^{2r}\tilde{R}_{n}^{\left( -2r\right)
}\left( g,x_{0}\right) =\frac{\left( -1\right) ^{r+1}}{2r \pi }\left[ g\left( x_{0}+0\right) -g\left( x_{0}-0\right) \right] ,
\end{equation}%
where $\tilde{R}_{n}\left( g,x\right) =\underset{k=n}{\overset{\infty }{\sum }}\left(
a_{k}\left( g\right) \sin kx-b_{k}\left( g\right) \cos kx\right) $ is the tail of the conjugate Fourier
series of $g$.

\item[b)] If $\Lambda $ is such that $\Lambda BV\supsetneqq HBV$, the
assertion a) does not hold for $\Lambda BV\setminus HBV$.
\end{enumerate}
\end{thm}

\section{Generalized Fourier-Jacobi and Fourier-Chebyshev series}

\subsection{Notation}

\ \ By $C^{p}\left[ -1,1\right] $, $p\in
%TCIMACRO{\U{2115} }%
%BeginExpansion
\mathbb{N}
%EndExpansion
_{0}$, we denote the space of $p-$times continuously differentiable
functions on $\left[ -1,1\right] $, where $C^{0}\left[ -1,1\right] \equiv C%
\left[ -1,1\right] $ is the space of continuous functions. Let $C^{-1}\left[
-1,1\right] $ denote the space of functions defined on $\left[ -1,1\right] $
which may have discontinuities only of the first kind and which are
normalized by the condition $f(x)=\left( f(x+0)+f\left( x-0\right) \right)
/2 $. If $f\in C^{-1}\left[ -1,1\right] $ has finitely many discontinuities,
let $M\equiv M(f)$ denote their number. By $x_{m}\equiv x_{m}(f)$ and $\left[
f\right] _{m}\equiv f(x_{m}+0)-$ $f(x_{m}-0)$, $m=1,\ldots ,M$, we denote
the points of discontinuity and the associated jumps of the function $f$. The $r$th derivative
of a function $f$ which piecewise belongs to $C^{p}\left[ -1,1\right] $, $%
p\geq r$, or which belongs to $C^{r-1}\left[ -1,1\right] $, is defined as $%
f^{(r)}(x)=\left( f^{(r)}(x+0)+f^{(r)}\left( x-0\right) \right) /2$,
whenever $f^{(r)}(x\pm 0)$ exist.

We say that $\mathbf{w}$\textbf{\ }is a \textit{generalized Jacobi weight},
i.e., $\mathbf{w}\in GJ$, if
\begin{gather*}
\mathbf{w}(t)=h(t)\left( 1-t\right) ^{\alpha }\left( 1+t\right) ^{\beta
}\left\vert t-\tilde{x}_{1}\right\vert ^{\delta _{1}}\cdots \left\vert
t-\tilde{x}_{N}\right\vert ^{\delta _{N}}\text{,} \\
h\in C\left[ -1,1\right] \text{, \ \ }h(t)>0\text{ \ }\left( \left\vert
t\right\vert \leq 1\right) \text{, \ \ }\omega \left( h;t;\left[ -1,1\right]
\right) t^{-1}\in L^{1}\left[ 0,1\right] \text{,} \\
-1<\tilde{x}_{1}<\cdots <\tilde{x}_{N}<1\text{, \ \ \ \ }\alpha \text{, }\beta \text{, }%
\delta _{1}\text{, \ldots , }\delta _{N}>-1\text{,}
\end{gather*}%
where $L^{1}\left[ 0,1\right] $ is the space of Lebesgue integrable
functions on $\left[ 0,1\right] $ and
\begin{equation*}
\omega \left( f;t;\left[ -1,1\right] \right) =\max \left\{ \left\vert
f(x)-f(y)\right\vert :x,y\in \left[ -1,1\right] \wedge \left\vert
x-y\right\vert \leq t\right\}
\end{equation*}%
is the modulus of continuity of $f\in C\left[ -1,1\right] $ on $\left[ -1,1%
\right] $. We always assume $\tilde{x}_{0}=-1$, and $\tilde{x}_{N+1}=1$. In
addition, for a fixed $\varepsilon \in \left( 0,\left( \tilde{x}_{\nu +1}-%
\tilde{x}_{\nu }\right) /2\right) $, $\nu =0,1,\ldots ,N$, we set $\Delta
\left( \nu ;\varepsilon \right) =\left[ \tilde{x}_{\nu }+\varepsilon ,\tilde{%
x}_{\nu +1}-\varepsilon \right] $.

By $\sigma \left( \mathbf{w}\right) =\left( P_{n}\left( \mathbf{w};x\right)
\right) _{n=0}^{\infty }$ we denote the system of algebraic polynomials $%
P_{n}\left( \mathbf{w};x\right) =\gamma _{n}(\mathbf{w})x^{n}+$ \textit{%
lower degree terms }with positive leading coeffients $\gamma _{n}(\mathbf{w}%
) $, which are orthonormal on $\left[ -1,1\right] $ with respect to the
weight $\mathbf{w}\in GJ$, i.e.,%
\begin{equation*}
\int_{-1}^{1}P_{n}\left( \mathbf{w};t\right) P_{m}\left( \mathbf{w};t\right)
\mathbf{w}(t)dt=\delta _{nm}\text{.}
\end{equation*}%
Such polynomials are called the \textit{generalized Jacobi polynomials}.

If $f\mathbf{w}\in L[-1,1]$, $\mathbf{w}\in GJ$, then $f$ has a Fourier
series with respect to the system $\sigma \left( \mathbf{w}\right) $, which
we will call \textit{generalized Fourier-Jacobi series}. Let $S_{n}(\mathbf{w%
};f;x)$ and $R_{n}(\mathbf{w};f;x)$ denote the $n$th partial sum and $n$th
order tail of the generalized Fourier-Jacobi series of $f$, respectively,
i.e.,%
\begin{gather*}
S_{n}(\mathbf{w};f;x)=\sum_{k=0}^{n-1}a_{k}\left( \mathbf{w};f\right)
P_{k}\left( \mathbf{w};x\right) =\int_{-1}^{1}f\left( t\right) K_{n}\left(
\mathbf{w};x;t\right) \mathbf{w}(t)dt\text{,} \\
R_{n}(\mathbf{w};f;x)=\sum_{k=n}^{\infty }a_{k}\left( \mathbf{w};f\right)
P_{k}\left( \mathbf{w};x\right) \text{,}
\end{gather*}%
where%
\begin{equation*}
a_{k}\left( \mathbf{w};f\right) =\int_{-1}^{1}f\left( t\right) P_{k}\left(
\mathbf{w};t\right) \mathbf{w}(t)dt
\end{equation*}%
is the $k$th Fourier coefficient of the function $f$, and%
\begin{equation*}
K_{n}\left( \mathbf{w};x;t\right) =\sum_{k=0}^{n-1}P_{k}\left( \mathbf{w}%
;x\right) P_{k}\left( \mathbf{w};t\right)
\end{equation*}%
is the Dirichlet kernel of the system $\sigma \left( \mathbf{w}\right) $.

When $h(t)\equiv 1$, $\left\vert t\right\vert \leq 1$, and $N=0$ (i.e., a
weight does not have singularities strictly inside the interval $\left(
-1,1\right) $), $\mathbf{w}\in GJ$ is called a Jacobi weight, and in this
case we use the commonly accepted notation "$\left( \alpha ,\beta \right) $"
instead of "$\mathbf{w}$" throughout. For example, we write $S_{n}^{(\alpha
,\beta )}(f;x)$ instead of $S_{n}(\mathbf{w};f;x)$ and the corresponding
series is called Fourier-Jacobi series. If $\alpha =\beta =-\frac{1}{2}$,
the corresponding Fourier-Jacobi series becomes Fourier-Chebyshev series.

\subsection{Equiconvergence}

\ \ We shall start with a simple proposition on convergence of the generalized Fourier-Jacobi series for functions of harmonic bounded variation.

\begin{prop}
\label{prop}
Let $f\in HBV$, $f\mathbf{w}\in L[-1,1]$, $\mathbf{w}\in GJ$. Then
$$\underset{n\rightarrow \infty }{\lim }S_{n}(\mathbf{w};f;x) = \frac{f(x+0)+f(x-0)}{2}$$
for every $x\in \left( -1,1\right) $, $x\neq \tilde{x}_{1},\ldots ,\tilde{x}_{N}$.

\begin{proof}
Let $S_{n}^{(-\frac{1}{2},-\frac{1}{2})}(f;x)$ be the $n$th partial sum of the Fourier-Chebyshev series of $f$. Kvernadze \cite[proof of
Theorem 7, p. 185]{KV1} proved uniform equiconvergence of Fourier-Chebyshev and generalized Fourier-Jacobi series for an arbitrary function $f\in HBV$
and a fixed $\varepsilon \in \left( 0,\frac{\tilde{x}_{\nu +1}-\tilde{x}_{\nu }}{2}\right) $%
, $\nu =0,1,2,\ldots ,N$,
\begin{equation}
\Vert S_{n}(\mathbf{w};f;x)-S_{n}^{(-\frac{1}{2},-\frac{1}{2})}(f;x)\Vert
_{C[\Delta (\nu ;\varepsilon )]}=o(1)\text{.}  \label{eq1}
\end{equation}%
Putting $x=\cos \theta$, $\theta\in \left( 0 , \pi \right) $, and $g(\theta )=f(\cos \theta )$, and taking into account that $g\left( \theta \mp 0\right) = f\left(
x\pm 0\right)$, we get
\begin{equation*}
S_{n}^{(-\frac{1}{2},-\frac{1}{2})}(f;x)=S_{n}(g,\theta) \rightarrow \frac{g(\theta+0)+g(\theta-0)}{2} = \frac{f(x+0)+f(x-0)}{2} \text{ as } n\rightarrow \infty
\end{equation*}
according to Waterman \cite[Theorem 2, p. 112]{WA}. For $x\neq \tilde{x}_{1},\ldots ,\tilde{x}_{N}$ there exist $\nu _{0}$ and $\varepsilon$ such that $x\in \left[ \tilde{x}_{\nu _{0} }+\varepsilon ,\tilde{x}_{\nu _{0} +1}-\varepsilon \right]$. Now, we have
\begin{gather*}
\left\vert S_{n}(\mathbf{w};f;x)-\frac{f(x+0)+f(x-0)}{2}\right\vert \leq \\
\left\vert S_{n}(\mathbf{w};f;x)-S_{n}^{(-1/2,-1/2)}(f;x)\right\vert
+\left\vert S_{n}^{(-1/2,-1/2)}(f;x)-\frac{f(x+0)+f(x-0)}{2}\right\vert  \\
\leq \left\Vert S_{n}(\mathbf{w};f;x)-S_{n}^{(-1/2,-1/2)}(f;x)\right\Vert _{C%
\left[ \Delta \left( \nu _{0};\varepsilon \right) \right] }+\left\vert
S_{n}^{(-1/2,-1/2)}(f;x)-\frac{f(x+0)+f(x-0)}{2}\right\vert \\
=o\left(1\right) \text{.}
\end{gather*}
\end{proof}
\end{prop}

\begin{cor}
\label{cor}
Let $f \in HBV$ and $\Delta \left( \nu;\varepsilon \right)$ be as above. Then
\begin{equation}
\Vert R_{n}(\mathbf{w};f;x)-R_{n}^{(-\frac{1}{2},-\frac{1}{2})}(f;x)\Vert
_{C[\Delta (\nu ;\varepsilon )]}=o(1)\text{.}  \label{eq2}
\end{equation}

\begin{proof}
For $x\in \left( -1,1\right) $, $x\neq \tilde{x}_{1},\ldots ,\tilde{x}_{N}$ Proposition \ref{prop} gives us

\begin{eqnarray*}
S_{n}(\mathbf{w};f;x) &=&\frac{f(x+0)+f(x-0)}{2}-R_{n}(\mathbf{w};f;x), \\
S_{n}^{(-\frac{1}{2},-\frac{1}{2})}(f;x) &=&\frac{f(x+0)+f(x-0)}{2}-R_{n}^{(-\frac{1}{2},-\frac{%
1}{2})}(f;x).
\end{eqnarray*}
This and \eqref{eq1} yield the assertion.
\end{proof}
\end{cor}

\subsection{Determination of a jump}

\ \ In order to prove an unconditional result on determination of a jump discontinuity of a function $f \in V_2$ by the tails of its integrated Fourier-Chebyshev series, we shall need the following lemma (cf. \cite[Remark, p. 236]{AV1}). For the sake of completeness of the argument, we include the proof of the Lemma.

\begin{lem}
\label{bbb} Let $f\in V_{2}$ be a $2\pi -$periodic function. Then, $%
n\sum_{k=n}^{\infty }\rho _{k}^{2}(f)=O\left( 1\right) $, where $\rho
_{k}^{2}(f)=a_{k}^{2}(f)+b_{k}^{2}(f)$ is the magnitude of the $k$th Fourier
coefficient.

\begin{proof}
If the Fourier series of $f$ is given by%
\begin{equation*}
f(x)\sim \frac{a_{0}}{2}+\sum_{m=1}^{\infty }a_{m}\cos mx+b_{m}\sin mx,
\end{equation*}%
then the Fourier series of $f(\cdot +t)$ reads
\begin{eqnarray*}
f(x+t) &\sim & \frac{a_{0}}{2}+\sum_{m=1}^{\infty }A_{m}\left( t\right) \cos
mx+B_{m}\left( t\right) \sin mx\text{,}
\end{eqnarray*}%
where $A_{m}\left( t\right) =a_{m}\cos mt+b_{m}\sin mt$ and $%
B_{m}\left( t\right) =b_{m}\cos mt-a_{m}\sin mt$. Thus,
\begin{equation*}
f(x+t)-f(x)\sim \sum_{m=1}^{\infty }\left( A_{m}\left( t\right)
-a_{m}\right) \cos mx+\left( B_{m}\left( t\right) -b_{m}\right) \sin mx\text{%
.}
\end{equation*}%
Simple calculations yield
\begin{equation*}
A_{m}\left( t\right) -a_{m}=2B_{m}\left( \frac{t}{2}\right) \sin \frac{mt}{2}%
\text{ and }B_{m}\left( t\right) -b_{m}=-2A_{m}\left( \frac{t}{2}\right)
\sin \frac{mt}{2}\text{.}
\end{equation*}%
Hence,%
\begin{equation*}
f(x+\frac{\pi }{n})-f(x)\sim 2\sum_{m=1}^{\infty }\left[ B_{m}\left( \frac{%
\pi }{2n}\right) \cos mx-A_{m}\left( \frac{\pi }{2n}\right) \sin mx\right]
\sin \frac{m\pi }{2n}\text{.}
\end{equation*}%
Parseval's identity gives us
\begin{equation*}
\frac{1}{\pi }\int_{0}^{2\pi }\left[ f\left( x+\frac{\pi }{n}\right)
-f\left( x\right) \right] ^{2}dx=4\sum_{m=1}^{\infty }\left[ A_{m}^{2}\left(
\frac{\pi }{2n}\right) +B_{m}^{2}\left( \frac{\pi }{2n}\right) \right] \sin
^{2}\frac{m\pi }{2n}\text{.}
\end{equation*}%
Since $A_{m}^{2}\left( t\right) +B_{m}^{2}\left( t\right)
=a_{m}^{2}+b_{m}^{2}=\rho _{m}^{2}$, the last equation becomes
\begin{equation*}
\frac{1}{\pi }\int_{0}^{2\pi }\left[ f\left( x+\frac{\pi }{n}\right)
-f\left( x\right) \right] ^{2}dx=4\sum_{m=1}^{\infty }\rho _{m}^{2}\sin ^{2}%
\frac{m\pi }{2n}\text{.}
\end{equation*}%
Due to the periodicity of $f$, we have%
\begin{equation*}
\frac{1}{\pi }\int_{0}^{2\pi }\left[ f\left( x+k\frac{\pi }{n}\right)
-f\left( x+(k-1)\frac{\pi }{n}\right) \right] ^{2}dx=4\sum_{m=1}^{\infty
}\rho _{m}^{2}\sin ^{2}\frac{m\pi }{2n}
\end{equation*}%
for every positive integer $k$.
Therefore,
\begin{equation*}
\sum_{k=1}^{2n}\frac{1}{\pi }\int_{0}^{2\pi }\left[ f\left( x+k\frac{\pi }{n}%
\right) -f\left( x+(k-1)\frac{\pi }{n}\right) \right] ^{2}dx=8n\sum_{m=1}^{%
\infty }\rho _{m}^{2}\sin ^{2}\frac{m\pi }{2n}\text{.}
\end{equation*}%
Changing the order of summation and integration on the left-hand side in the above equation and taking into account that $f\in V_2$, we get
\begin{equation*}
n\sum_{m=1}^{\infty }\rho _{m}^{2}\sin ^{2}\frac{m\pi }{2n}=O(1)\text{.}
\end{equation*}%
Now,
\begin{equation*}
n\sum_{k=1}^{\infty }\rho _{k}^{2}\sin ^{2}\frac{k\pi }{2n}\geq
n\sum_{k=1}^{n}\rho _{k}^{2}\sin ^{2}\frac{k\pi }{2n}\geq
n\sum_{k=1}^{n}\rho _{k}^{2}\left( \frac{2}{\pi }\cdot \frac{k\pi }{2n}%
\right) ^{2}=\frac{1}{n}\sum_{k=1}^{n}k^{2}\rho _{k}^{2}\text{.}
\end{equation*}%
Thus,%
\begin{equation*}
\frac{1}{n}\sum_{k=1}^{n}k^{2}\rho _{k}^{2}=O(1)\text{.}
\end{equation*}%
Using Abel's partial summation formula, we get%
\begin{eqnarray*}
\sum_{k=n}^{m}\rho _{k}^{2} &=&\sum_{k=n}^{m}\frac{1}{k^{2}}\left( k^{2}\rho
_{k}^{2}\right) =\frac{1}{m^{2}}\sum_{i=n}^{m}i^{2}\rho
_{i}^{2}+\sum_{k=n}^{m-1}\left( \frac{1}{k^{2}}-\frac{1}{(k+1)^{2}}\right)
\sum_{i=n}^{k}i^{2}\rho _{i}^{2} \\
&=&O(1)\left[ \frac{1}{m}\cdot \frac{1}{m}\sum_{i=n}^{m}i^{2}\rho
_{i}^{2}+\sum_{k=n}^{m-1}\left( \frac{1}{k}-\frac{1}{k+1}\right) \frac{1}{k}%
\sum_{i=n}^{k}i^{2}\rho _{i}^{2}\right]  \\
&=&O(1)\left[ \frac{1}{m}+\sum_{k=n}^{m-1}\left( \frac{1}{k}-\frac{1}{k+1}%
\right) \right] =O\left(\frac{1}{n}\right)
\end{eqnarray*}%
for arbitrary positive integer $m>n$. Hence,
\begin{equation*}
n\sum_{k=n}^{\infty }\rho _{k}^{2}=O(1)\text{.}
\end{equation*}
\end{proof}
\end{lem}

\ \ Now, we can turn our attention to determination of jump discontinuities by means of the tails of integrated Fourier-Chebyshev series.

\begin{thm}
\label{ccc} \ \newline
\begin{enumerate}
\item[a)]If $f\in HBV$ has finitely many discontinuities, then
\begin{equation}
\underset{n\rightarrow \infty }{\lim }n\left[ R_{n}^{(-\frac{1}{2},-\frac{1%
}{2})}(f;x)\right] ^{(-1)}=-\frac{(1-x^{2})^{\frac{1}{2}}}{\pi }%
(f(x+0)-f(x-0)) \label{alfa}
\end{equation}%
is valid for each fixed $x\in \left( -1,1\right) $, where
\begin{equation*}
\left[ R_{n}^{(-\frac{1}{2},-\frac{1%
}{2})}(f;x)\right] ^{(-1)}=\int_{-1}^{x}R_{n}^{(-\frac{1}{2},-\frac{1%
}{2})}(f;y) \ d y \text{.}
\end{equation*}

\item[b)]If $f\in V_{2}$,
then the relation \eqref{alfa} holds true without restriction on the number of
discontinuities.
\end{enumerate}

\begin{proof}
Integrating $R_{n}^{(-\frac{1}{2},-\frac{1}{2})}(f;y)$ on $\left[ -1,x\right]
$ and using the identity
\begin{equation*}
R_{n}^{(-\frac{1}{2},-\frac{1}{2})}(f;y)=R_{n}(g,\theta )\text{,}
\end{equation*}%
where $y=\cos \theta $, we get
\begin{equation}
\begin{aligned} \lbrack R_{n}^{(-\frac{1}{2},-\frac{1}{2})}(f;x)]^{(-1)}&
=\int_{\arccos x}^{\pi }R_{n}(g,\theta )\sin \theta \ d\theta \\ & =\left. \left[
\sin \theta R_{n}^{(-1)}(g;\theta )\right] \right\vert _{\arccos x}^{\pi
}-\int_{\arccos x}^{\pi }R_{n}^{(-1)}(g;\theta )\cos \theta \ d\theta \\ &
=-\sin \eta \ R_{n}^{(-1)}(g;\eta)-\int_{\eta}^{\pi }R_{n}^{(-1)}(g;\theta )\cos \theta \ d\theta  \\ &
=-(1-x^{2})^{\frac{1}{2}}R_{n}^{(-1)}(g;\eta)-\int_{\eta}^{\pi
}R_{n}^{(-1)}(g;\theta )\cos \theta \ d\theta \text{,} \end{aligned}
\label{eq3}
\end{equation}
where we put $\eta = \arccos x$.

a) Any $g\in HBV$ with $M$ points of discontinuity can be represented in the following form
\begin{equation}
g\equiv g_{c}+\frac{1}{\pi }\sum_{m=1}^{M}[g]_{m}G(\theta _{m};.),
\label{cont}
\end{equation}%
where $G(\theta )=\displaystyle\frac{\pi -\theta }{2}$, $\theta \in (0,2\pi
) $, is a $2\pi -$periodic sawtooth function, $\theta _{m}$ and $[g]_{m}$,$%
\;m=1,2,\ldots,M$, are the points of discontinuities and the associated jumps of
the function $g$, respectively, and $G(\theta _{m};\theta)=G(\theta - \theta _{m})$. The function $g_{c}$ is a $2\pi -$periodic continuous function, which is piecewise smooth on $[-\pi ,\pi ]$.

From $G(\theta)=\sum_{n=1}^{\infty }\frac{\sin n\theta}{n}$, we obviously have $R_{n}^{\left( -1\right) }\left( G;\theta \right) =O\left( \frac{1}{n}\right)$ and
\begin{equation} \label{eq06}
nR_{n}^{\left( -1\right) }\left( G\left( \theta _{m};\cdot \right) ;\theta
\right) =O\left( 1\right) \text{ uniformly, }m=1,\ldots ,M\text{.}
\end{equation}

Now, $g_{c}\in C\cap HBV$. Fourier series of $g_{c}$ converges uniformly by a theorem of Waterman \cite[Theorem 2, p. 112]{WA}. Since%
\begin{equation*}
R_{n}^{(-1)}(g_{c};\theta )=\int R_{n}(g_{c};\theta )\ d\theta
\end{equation*}%
and $R_{n}(g_{c};\theta )$
converges uniformly on $\left[ -\pi ,\pi \right] $, then
\begin{equation}
nR_{n}^{(-1)}(g_{c};\theta )=o\left( 1\right) \text{\ uniformly}  \label{integ1}
\end{equation}
by a theorem of Tong \cite[Theorem, p. 252]{To}. Combining (\ref{cont}), (\ref{eq06}) and (\ref{integ1}), we get
\begin{equation}
nR_{n}^{(-1)}(g;\theta )=O\left( 1\right) \text{\ uniformly.}  \label{integ5}
\end{equation}
If $\theta$ is a point of continuity of the function $g$, Theorem \ref{aaa} implies $nR_{n}^{(-1)}(g;\theta ) \rightarrow 0$ as $n \rightarrow \infty$. Therefore,
\begin{equation*}
\underset{n\rightarrow \infty }{\lim } nR_{n}^{(-1)}(g;\theta ) \cos \theta = 0
\end{equation*}
everywhere except at a finite set of discontinuities of $g$. Applying the Lebesgue dominated convergence theorem \cite[p. 267]{ROY}, we obtain
\begin{equation}
\underset{n\rightarrow \infty }{\lim }\int_{\arccos x}^{\pi }n
R_{n}^{(-1)}(g;\theta ) \cos \theta \ d\theta =0\text{.}  \label{integ3}
\end{equation}%
Multiplying (\ref{eq3}) by $n$, letting $n\rightarrow \infty $, using (\ref%
{integ3}) and Theorem \ref{aaa} with $r=0$ and taking into account that $f\left(
x\pm 0\right) =g\left( \theta \mp 0\right) $, we get%
\begin{equation*}
\underset{n\rightarrow \infty }{\lim }n\left[ R_{n}^{(-\frac{1}{2},-\frac{1%
}{2})}(f;x)\right] ^{(-1)}=-\frac{(1-x^{2})^{\frac{1}{2}}}{\pi }%
(f(x+0)-f(x-0))\text{.}
\end{equation*}

b) For $g\in V_{2}$, applying the Cauchy-Schwartz inequality and Lemma \ref{bbb}, we get
\begin{eqnarray*}
n\left\vert R_{n}^{(-1)}(g;\theta )\right\vert  &\leq &n\sum_{k=n}^{\infty }%
\frac{\left\vert a_{k}(g)\right\vert +\left\vert b_{k}(g)\right\vert }{k} \\
&\leq &\sqrt{2}n\left( \sum_{k=n}^{\infty }\left(
a_{k}^{2}(g)+b_{k}^{2}(g)\right) \right) ^{1/2}\left( \sum_{k=n}^{\infty }%
\frac{1}{k^{2}}\right) ^{1/2} \\
&=&\sqrt{2}nO\left( \frac{1}{\sqrt{n}}\right) O\left( \frac{1}{\sqrt{n}}%
\right) =O(1)\text{,}
\end{eqnarray*}
i.e., $nR_{n}^{(-1)}(g;\theta )=O\left( 1\right) $ uniformly. The Lebesgue dominated convergence theorem yields
\begin{equation*}
\underset{n\rightarrow \infty }{\lim }\int_{\arccos x}^{\pi }n \
R_{n}^{(-1)}(g;\theta ) \cos \theta \ d\theta = \int_{\arccos x}^{\pi }\underset{n\rightarrow \infty }{\lim } n \ R_{n}^{(-1)}(g;\theta ) \cos \theta \ d\theta\text{.}
\end{equation*}
As already noticed, if $\theta$ is a point of continuity of the function $g$, Theorem \ref{aaa} implies 
$$nR_{n}^{(-1)}(g;\theta ) \rightarrow 0 \text{ as } n \rightarrow \infty.$$
Therefore, $\underset{n\rightarrow \infty }{\lim } nR_{n}^{(-1)}(g;\theta ) \cos \theta = 0$ everywhere except at a denumerable set of discontinuities of $g$. Thus, (\ref{integ3}) and consequently (\ref{alfa}) hold true for $g\in V_{2}$ without finiteness restriction on the number of discontinuities of $g$.
\end{proof}
\end{thm}

\begin{rem}
Theorem \ref{ccc} transfers a corresponding result by Kvernadze, Hagstrom
and Shapiro \cite{KHS} from the trigonometric case to the setting of Fourier-Chebyshev series. At the same time, it generalizes their result in two directions. If the finiteness assumption on the number of
discontinuities of a function is kept, then we can deal with the whole class
$HBV$, as demonstrated in part a) of the proof. On the other hand, if the
attention is restricted to the subclass $V_2$, then part b) shows that the
finiteness assumption can be removed.
\end{rem}

\begin{rem}
In view of Theorem \ref{AVth} above, the part b) of Theorem \ref{ccc} is obviously valid for the Watermann class $\left\{ n^{\frac{1}{2} }\right\} BV$ and Chanturiya's classes $V\left[n^{\alpha }\right]$, $0<\alpha<\frac{1}{2}$.
\end{rem}

\bigskip


%% \bibliographystyle{elsarticle-num}
%% \bibliography{<your-bib-database>}






% \end{thebibliography}

\end{document}